\documentclass{amsart}

\usepackage{amssymb}
\usepackage{amsmath}
\usepackage{enumerate}
\usepackage{array}
\usepackage{todonotes}
\usepackage{color}
\usepackage{dsfont} 
\usepackage{comment} 
\usepackage{pgfplots} 
\usetikzlibrary{patterns} 

\newtheorem{thm}{Theorem}
\newtheorem{lem}{Lemma}

\newtheorem{prop}{Proposition}

\theoremstyle{remark}

\newcommand{\x}[2]{f_\omega^{#2}({#1})}  
\newcommand{\xo}[2]{f_\omega^{#2}({#1})} 
\newcommand{\mcP}{\mathcal P_{M, \alpha}}
\newcommand{\mbP}[1]{\mathbb{P}_{#1}}
\newcommand{\mbE}[1]{\mathbb{E}_{#1}}

\newcommand{\red}[1]{\textcolor{red}{#1}}
\newcommand{\vare}{\varepsilon}
\newcommand{\f}[3]{f_{{#1}_{#2}}\circ\ldots\circ f_{{#1}_1}({#3})}
\newcommand{\dxo}[1]{|f_\omega^{#1}(x)-f_\omega^{#1}(y)|}
\newcommand{\tim}[1]{T_{#1}^{h}(x)}
\newcommand{\del}[1]{\Delta^{x,h}_{#1}}

\title{Alsed\`a-Misiurewicz systems with place-dependent probabilities}
\author{Klaudiusz Czudek}
\address{Klaudiusz Czudek, Institute of Mathematics Polish Academy of Sciences,
\'Sniadeckich 8,  00-656 Warszawa, Poland}
\email{klaudiusz.czudek@gmail.com}



\thanks{The research was supported by the Polish Ministry of Science and Higher Education "Diamond Grant" 0090/DIA/2017/46.}

\begin{document}
\maketitle

\begin{abstract}
We consider systems of two specific piecewise linear homeomorphisms of the unit interval, so called the Alsed\`{a}-Misiurewicz systems, and investigate the basic properties of Markov chains which arise when these two transformations are applied randomly with probabilities depending on the point of the interval. Though this iterated function system is not contracting in average and known methods do not apply, stability and the strong law of large numbers are proven.
\end{abstract}


\section{Introduction}
\subsection{The main results}
Let $f_0$ be an interval homeomorphisms such that its graph consists of two straight lines, the first one connecting $(0,0)$ with some point $(x_0, y_0)\in (1/2,1)\times [1/2,1)$ under diagonal, and the second one connecting $(x_0, y_0)$ with $(1, 1)$. Next, let $f_1$ be the interval homeomorphisms defined by $f_1(x)=1-f_0(1-x)$, $x\in [0,1]$ (see Figure 1). Setting $a_0=\frac{y_0}{x_0}$ and $a_1=\frac{1-y_0}{1-x_0}$ we can write
$$
f_0(x) := \left\{ \begin{array}{ll}
a_0x & \textrm{if $x\le x_0$}\\
a_1(x-1)+1 & \textrm{if $x>x_0$}
\end{array} \right. \qquad  \textrm{and} \qquad f_1(x) :=1-f_0(1-x).$$
Fix two positive real functions $p_0, p_1$ on $[0,1]$ with $p_0(x)+p_1(x)=1$ for every $x\in [0,1]$. It defines the following natural random process on the interval $(0,1)$: being at any point $x\in (0,1)$ we choose transformation $f_i$ with probability $p_i(x)$ and move to the point $f_i(x)$, $i=0,1$. To be more strict, we define the family of Markov chains with common transition probabilities given by the formula $\pi(x,\cdot):=p_0(x)\delta_{f_0(x)}+p_1(x)\delta_{f_1(x)}$, $x\in(0,1)$.

As far as we know, stability of these Markov chains, i.e. arising from random application of transformations on the interval, were always proved under assumption that all transformations are contractions or are contracting in average. However, recently several papers have been published which established stability with dropping this assumptions. The first one and probably the most important for us was by Lluis Alsed\`a and Micha\l{} Misiurewicz \cite{Alseda-Misiurewicz} in 2014 where the authors showed that if we consider two transformations defined in the first paragraph for $y_0=1/2$ and choose it randomly with constant and equal probabilities $1/2, 1/2$ then the corresponding Markov chain is stable. After \cite{Baranski-Spiewak} we call the systems defined above the Alsed\`a-Misiurewicz systems (in \cite{Baranski-Spiewak} the only restriction for $(x_0,y_0)\in (0,1)\times (0,1)$ is that it should be under diagonal). Later their results were generalized to the case of two $C^2$ diffeomorphisms (\cite{Homburg}) or even arbitrary finite number of homeomorphisms (\cite{Szarek_Zdunik}) satisfying additional assumptions that from each point we have positive probabilities of moving to the left and moving to the right, all functions are differentiable at 0 and 1 and the average Lyapunov exponents at these points are positive. In our setting we define the average Lyapunov exponents at $0$ and $1$ by the formulae
\begin{equation}
\label{I1}
\begin{gathered}
\Lambda_0:=p_0(0)\log(a_0)+p_1(0)\log(a_1),\\
\Lambda_1:=p_0(1)\log(a_1)+p_1(1)\log(a_0).
\end{gathered}
\end{equation}
In the general case coefficients $a_0, a_1$ should be replaced by derivatives at $0$ and $1$, respectively. Note that all known results are proven under assumption that probability of choice of a transformation does not depend on the point of the interval.

The most important papers concerning systems with place-dependent probabilities are probably \cite{BDEG}, \cite{Lasota} where the stability of the corresponding Markov chains is proved under the most general assumptions in arbitrary locally compact metric spaces. However, one of them is contractivity in average which in our case is never satisfied, therefore we can neither apply the result, nor use the proof. In \cite{BDEG} and \cite{Stenflo} one can find further references and historical comments.


The goal of our paper is to provide proofs of ergodicity, stability and the strong law of large numbers for Alsed\'a-Misiurewicz systems in the case when probabilities are not necessarily constant. To this end we introduce the following assumptions:
\begin{enumerate}[({A}1)]
\item $\frac 1 2<x_0<1$ and $\frac 1 2\le y_0<x_0$
\item $p_0, p_1$ are Dini continuous,
\item $0<p_i(x)<1$ for $x\in[0,1]$ and $i=0,1$,
\item $\Lambda_0, \Lambda_1>0$.
\end{enumerate}
The functions $p_0, p_1$ are Dini continuous, which means that for every $C\ge 0$ and $t<1$ we have $\sum_n \beta(Ct^n)<\infty$, where $\beta$ denotes the modulus of continuity of $p_0, p_1$, i.e.
$$\beta(t):=\max_{ i=0,1} \sup_{x\in (0,1), |h|\le t} |p_i(x)-p_i(x+h)|.$$ 
We do not need any further assumptions on contractiveness of the system. Our two main results are the following theorems.

\begin{thm}
If (A1)-(A4) hold then there exists a unique Borel probability measure $\mu_*\in \mathcal M$ such that the Markov chain $(X^{\mu_*}_{n})$ is stationary.
\end{thm}

\begin{thm}
\label{thm:C}
If (A1)-(A4) hold, $\nu$ is any Borel probability measure then the Markov chain $(X^{\nu}_n)$ is asymptotically stable.
\end{thm}

The last theorem was proved in the case of general Markov chains on compact spaces in \cite{Breiman}. Later it was proved in \cite{Diaconis} in the case of iterated systems  of contractions in $\mathbb{R}^n$ with constant probabilities and in \cite{Elton} in the case of systems contracting in average with place dependent probabilities on locally compact spaces. Our system does not satisfy assumptions of any of these theorems, however, using some ideas from the last paper we are still able to prove it. In its statement it is essential that it holds for every point $x\in(0,1)$, not only for $\mu_*$ almost every.

\begin{thm}[The Strong Law of Large Numbers]
If (A1)-(A4) hold, $x\in (0,1)$, $\varphi\in C\big((0,1)\big)$ then
$$\frac {\varphi(X^{x}_{1})+\ldots+\varphi(X^{x}_{n})}{n}\to \int\varphi\textrm{d$\mu_*$} \qquad \textrm{a.s.}$$
\end{thm}

\section{Notation}

\begin{center}

\begin{tikzpicture}
\draw (0,0) rectangle (6,6); 
\draw (0,0) -- (6,6); 

\draw (4.5,0) node[below] {$x_0$};
\draw (1.5,0) node[below] {$1-x_0$};
\draw (6,3.3) node[right] {$y_0$};
\draw (0,2.7) node[left] {$1-y_0$};

\draw[red] (0,0) --  (4.5, 3.3);
\draw[red] (4.5,3.3) -- (6,6);

\draw[red] (0,0) -- (1.5, 2.7);
\draw[red] (1.5, 2.7) -- (6, 6);

\draw (3,1.9) node {$f_0$};
\draw (3,4.1) node {$f_1$};

\draw[dashed] (1.5,0) -- (1.5,2.7);
\draw[dashed] (4.5, 0) -- (4.5,3.3);
\draw[dashed] (0,2.7) -- (1.5,2.7);
\draw[dashed] (4.5,3.3) -- (6,3.3);
\draw[dashed] (0,3) -- (6,3);

\draw[pattern=north east lines, pattern color=blue!20!white] (3,3) -- (6,3) -- (6,6) -- cycle;

\node [below=0.8cm, align=flush center,text width=8cm] at (3,0)
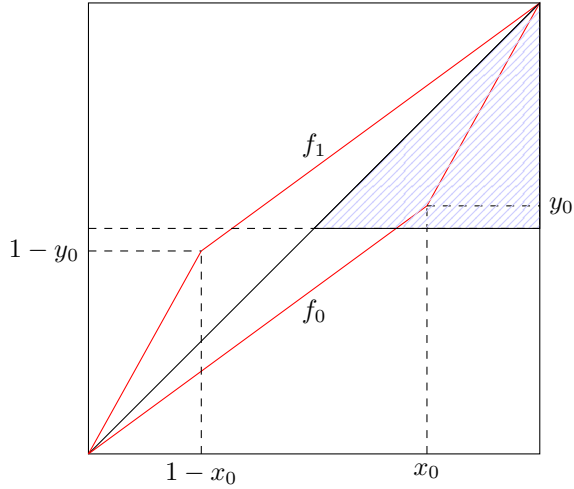
{Figure 1. The example of Alsed\`{a}-Misiurewicz system. The hatched area is the set of points $(x,y)$ which satisfy assumption (A1).};

\end{tikzpicture}
\end{center}

The space of Borel probability measures on $(0,1)$ will be denoted by $\mathcal M_1$ and the space of all positive Borel measures by $\mathcal M$. Recall that the family of transition probabilities $p(x,\cdot )\in\mathcal M_1$, $x\in (0,1)$ by the formula
$$p(x,\cdot):= p_0(x)\delta_{f_0(x)}+p_1(x)\delta_{f_1(x)} \quad \textrm{for $x\in (0,1)$}.$$
Let us choose an initial distribution $\mu\in \mathcal M_1$. Together with the transition probabilities it defines the Markov chain $(X^{\mu}_n)$ on $(0,1)$. For simplicity of notation, we shall write $(X^x_n)$ when $\mu=\delta_x$. Let us stress that values of this Markov chain are in the open interval $(0,1)$, not $[0,1]$.

The canonical space for this Markov chain is constructed as follows. Put $\Omega=(0,1)^\infty$, $\mathcal G=\mathcal B (0,1)^\infty$. Here $\mathcal B (0,1)$ stands for the $\sigma$-algebra of Borel subsets of $(0,1)$. We define the family of measures $\mathbb{P}^\infty_x, x\in (0,1)$ on $(\Omega,\mathcal G)$ by giving its values on cylinder sets, i.e.
$$\mathbb{P}^\infty_x(A_1\times\ldots\times A_k\times (0,1)^\infty):=\int_{A_1} p(x,dx_1)\int_{A_2} p(x_1, dx_2)\ldots\int_{A_k}p(x_{k-1},dx_k),$$
where $A_1,\ldots, A_k\in \mathcal B(0,1)$, $x\in (0,1)$. Existence of the unique extension to a measure on $\mathcal G$ follows from the Kolmogorov Extension Theorem. Fix the initial distribution $\nu\in \mathcal M_1$ and define the measures $\mathbb{P}^\infty_\nu$ on cylinders by
$$\mathbb{P}^\infty_\nu (A\times B):=\int_{A} \mathbb{P}^\infty_x(B)\nu(dx),$$
for $A\in\mathcal B(0,1), B\in\mathcal G$. This measure the unique extension to $\mathcal G$ by the Kolmogorov Extension Theorem. Now the sequence $(\pi_n)$ of projections defined on $(\Omega, \mathcal{G}, \mathbb{P}_\nu^\infty)$ by $\pi_n(x_1, x_2, \ldots):=x_n$, $n\ge 1$, is the canonical realization of the Markov chain $(X^\nu_n) $.

The processes $(X^x_n), x\in (0,1)$ may be also realized on the space $\Sigma=\{0,1\}^\mathbb{N}$ with the standard product $\sigma$-algebra $\mathcal F$ and the probability measure $\mathbb{P}_x$ defined on cylinders $C_{i_1,\ldots,i_k}=\{\omega\in\Sigma: \omega_1=i_1,\ldots,\omega_k=i_k\}$ by
$$\mathbb{P}_x(C_{i_1,\ldots,i_k}):=p_{i_1}(x)p_{i_2}(f_{i_1}(x))\ldots p_{i_n}(f_{i_{n-1}}\circ\ldots\circ f_{i_1}(x)).$$
Then it is clear that $f_\omega^n(x):=f_{\omega_{n}}\circ\ldots\circ f_{\omega_1}(x)$, where $\omega=(\omega_1, \omega_2,\ldots)$ is a realization of $(X_n^x)$. Expectation with respect to $\mathbb{P}_x$ is denoted by $\mathbb{E}_x$. By $\theta_n$ we denote the shift $\theta_n:\Sigma\rightarrow\Sigma$, $\theta_n(\omega):=(\omega_{n+1}, \omega_{n+2},\ldots)$, where $\omega=(\omega_1, \omega_2,\ldots)$. For $n\ge 1$ and $\omega\in\Sigma$ put
$$a^n_\omega:=a_{\omega_n}\ldots a_{\omega_1}.$$


In order to describe the evolution of $(X^\nu_n)$ we introduce the Markov-Feller operator $P:\mathcal M\rightarrow \mathcal M$ by
$$P\mu(A):=\int_{f_0^{-1}(A)}p_0(x)\mu(dx)+\int_{f_1^{-1}(A)}p_1(x)\mu(dx),$$
for $A\in\mathcal B(0,1), \mu\in\mathcal M$. Its predual operator $U: C(0,1)\rightarrow C(0,1)$ is given by
$$U\varphi(x):=p_0(x)\varphi(f_0(x))+p_1(x)\varphi(f_1(x)),$$
for $\varphi\in C(0,1)$ and $x\in (0,1)$. By "predual" we mean that
$$\int_{(0,1)}\varphi dP\mu=\int_{(0,1)}U\varphi d\mu$$
for every $\mu\in\mathcal M$ and $\varphi\in C(0,1)$. The operator $P$ is linear, i.e. $P(\lambda_1\mu_1+\lambda_2\mu_2)=\lambda_1P\mu_1+\lambda_2P\mu_2$ for $\lambda_1,\lambda_2\ge 0$, $\mu_1,\mu_2\in \mathcal M$. It also preserves the total mass of a measure, i.e. $P\mu\big((0,1)\big)=\mu\big((0,1))\big)$ for $\mu\in\mathcal M$. We say that a measure $\mu_*\in\mathcal M$ is invariant for the operator $P$ if $P\mu_*=\mu_*$. In that case we say that the operator $P$ is asymptotically stable if $P^n\nu\to \mu_*$ weakly for every $\nu\in\mathcal M_1$.

The Markov-Feller operator $P$ has the property that the distribution of $X^\mu_n$ is $P^n\mu$ for all $n\ge 0$ and $\mu\in\mathcal M$. Therefore one can choose an initial distribution $\mu\in \mathcal M_1$ in such a way that $(X^\mu_n)$ is stationary if and only if $\mu$ is $P$-invariant and the Markov chain $(X_n^\mu)$ is stable if and only if $P$ is asymptotically stable.

Following \cite{Homburg} we define
$$\mathcal P_{M, \alpha}:=\{ \mu\in \mathcal M_1 : \mu((0,x))\le Mx^\alpha \ \textrm{and} \ \mu((1-x,1))\le Mx^\alpha \ \textrm{for all $x\in(0,1)$}  \}.$$
By what we just mentioned, the theorem below is equivalent to the existence of a unique invariant probability measure for the Markov-Feller operator $P$.

%
%

\section{The proof of Theorem 1}

\begin{proof}[Proof of existence]
The proof follows the lines of the proof from \cite{Homburg} with necessary changes. Namely, we shall show that there exist parameters $M\ge 1, \alpha\in (0,1)$ such that the class $\mathcal P_{M, \alpha}$ is invariant under the operator $P$. It is sufficient since in that case one can apply the standard Krylov-Bogoliubov technique, i.e. take any $\nu\in\mcP$ and define $\nu_n=\frac 1 n(\nu+\ldots+P^{n-1}\nu)$. By the $P$-invariance of $\mcP$, all $\nu_n$'s are in $\mcP$, and by weak-$\ast$ compactness of $\mcP$ there exists an accumulation point $\mu_*\in\mcP$ of this sequence which is an invariant measure. Details are left to the reader. What remains to show is the existence of parameters $M,\alpha$ with the desired property.

By the continuity of $p_0, p_1$ at the boundary, (A4) and (\ref{I1}) one can find $0<\varepsilon<1-x_0$ such that
\begin{equation}
\label{A1}
\begin{gathered}
\max_{t\le \varepsilon} p_0(t)\log a_0+\max_{t\le \varepsilon} p_1(t)\log a_1>\frac{\Lambda_0}{2},\\
\max_{t\le \varepsilon} p_0(1-t)\log a_1+\max_{t\le \varepsilon} p_1(1-t)\log a_0>\frac{\Lambda_1}{2}.
\end{gathered}
\end{equation}
Writing the Taylor formula of the function $\alpha\longmapsto a^{-\alpha}$ at 0 we obtain $a^{-\alpha}=1-\alpha\log a+o(\alpha)$, where $a$ is any fixed positive number. By this formula and (\ref{A1}) one can find  $\alpha\in(0,1)$ and $p\in(0,1)$ with
\begin{equation}
\label{A2}
\begin{gathered}
\max_{t\le \varepsilon} p_0(t) a_0^{-\alpha}+\max_{t\le \varepsilon} p_1(t)a_1^{-\alpha}<p,\\
\max_{t\le \varepsilon} p_0(1-t)a_1^{-\alpha}+\max_{t\le \varepsilon} p_1(1-t)a_0^{-\alpha}<p.
\end{gathered}
\end{equation}
Eventually, put $M$ to be any number greater or equal than $(a_0\varepsilon)^{-\alpha}>\varepsilon^{-\alpha}>1$.

We are in position to show the invariance of $\mcP$ for $M, \alpha$ chosen above. Take $\mu\in \mcP$ and $x\in (0,1)$. If $x\ge a_0\varepsilon$, then $Mx^\alpha\ge M(a_0\varepsilon)^\alpha\ge 1$, hence the condition $P\mu((0,x))\le Mx^\alpha$ is trivially satisfied. If $x<a_0\varepsilon$, then also $x<1-x_0$ and
$$P\mu((0,x))=\int_{(0,a_0^{-1}x]}p_0(t)\mu(dt)+\int_{(0,a_1^{-1}x]}p_1(t)\mu(dt)$$
$$\le \max_{t\le \varepsilon} p_0(t) M a_0^{-\alpha}x^\alpha+\max_{t\le \varepsilon} p_1(t) M a_1^{-\alpha}x^\alpha< Mx^\alpha p <Mx^\alpha,$$

\noindent where in the last line we used (\ref{A2}). Therefore $P\mu((0,x))\le Mx^\alpha$. The proof that $P\mu((1-x,1))\le Mx^\alpha$ is analogous. The invariance of $\mcP$ is established.
\end{proof}



\begin{center}

\begin{tikzpicture}

\draw (0,0) -- (6,0); 

\draw (4.5,0) node[below] {$x_0$};
\draw (1.5,0) node[below] {$1-x_0$};
\draw (3.6,0) node[above] {$y_0$};
\draw (4.3,0) node[above] {$f_1(y_0)$};
\draw (2.4,0) node[above] {$1-y_0$};
\draw (0,0) node[below] {$0$};
\draw (6,0) node[below] {$1$};
\draw (3,0) node[below] {$1/2$};

\foreach \Point in {(0,0), (4.5,0), (1.5,0), (3.6,0), (2.4,0), (6,0), (4.3,0), (3,0)}{
    \node at \Point {\textbullet};
    }

\node [below=0.8cm, align=flush center,text width=8cm] at (3,0)
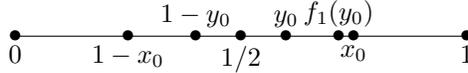
{Figure 2. The order of the points $1-x_0$, $1-y_0$, $y_0$, $x_0$. };

\end{tikzpicture}
\end{center}

\noindent
We are now going to make some use of (A1). Take $\eta_1$ such that the following condition is satisfied
\begin{equation}
\label{A3}
a_0 y-a_1(y-a_1\eta_1)<0 \ \textrm{for $y\ge 1-y_0$}.
\end{equation}
There exists such $\eta_1$. Indeed, since $a_0<1<a_1$, the linear function $y\longmapsto a_0y-a_1(y-a_1\eta_1)$ is decreasing and, in consequence, it suffices to show that there exists such $\eta_1$ for $y=1-y_0$. But since
$$a_0(1-y_0)-a_1((1-y_0)-a_1\eta_1)= (1-y_0)(a_0-a_1)+a_1^2\eta_1,$$
it just follows by $(1-y_0)(a_0-a_1)<0$. Let us also assume that $\eta_1$ is less than the  length of the interval $[1-x_0, 1-y_0]$ and satisfies
\begin{equation}
\label{A4}
f_1(y_0+a_1\eta_1)<x_0.
\end{equation}
This is possible by the continuity of $f_1$ and
\begin{equation}
\label{A5}
f_1(y_0)<x_0.
\end{equation}
To show this, however, we compute $f_1(y_0)=\frac{-y_0(1-y_0)}{x_0}+1$ and obtain that (\ref{A5}) is equivalent to the condition $y_0(1-y_0)>x_0(1-x_0)$. By the assumptions made on $x_0, y_0$ we have $1/2\le y_0<x_0$, so our statement follows from the monotonicity of the function $\psi(t):=t(1-t)$ on $[1/2, 1]$.


\begin{prop}
\label{prop:A3}
If $x, y\in [1-x_0, x_0]$ and $|x-y|<\eta_1$, then
$$|\xo{x}{n}-\xo{y}{n}|\le a_1|x-y|$$
for every $n$.
\end{prop}

In order to simplify the reasoning we assume that $x<y$ and $\omega$ such that $f^n_{\omega}(x)$ visits $(0,1-x_0)$ infinitely often and $f^n_{\omega}(y)$ visits $(x_0 ,1)$ infinitely often. In the end of the proof we will give a simple explanation that this assumption may be dropped.

\begin{lem}
\label{lem:A5}
If $1-x_0\le x<y\le x_0$, $|x-y|<a_1\eta_1$, $y>1-y_0$ and $u$ is such that $\xo{y}{n}\le x_0$ for all $n\le u$, then
$$|\xo{x}{n}-\xo{y}{n}|\le |x-y|$$
for all $n\le u$.
\end{lem}
\begin{proof}[Proof of Lemma 1]
Let $t$ be the least integer for which $f^t_\omega(x)<1-x_0$ and let $s<t$ be the maximal integer for which $f^s_\omega(y)>1-y_0$. Obviously $|f^n_\omega(x)-f^n_\omega(y)|\le |x-y|$ for $n\le s$, since both $f_0, f_1$ are contractions on $[1-x_0, x_0]$. Moreover, $f^s_\omega(x)$, $f^s_\omega(y)$ again satisfy assumptions of the lemma, therefore we may assume $s=0$. Next, define $r$ to be the moment of the first visit of $f^r_\omega(y)$ in $(1-y_0, 1)$. If we will show the claim for $n\le r$, then the points $f^r_\omega(x)$, $f^r_\omega(y)$ again satisfy the assumptions of the lemma, therefore we may assume $r=u$.

For this purpose observe that $\xo{y}{n}=a^n_\omega y$ and $\xo{x}{n}=a^n_\omega x$ for $n\le r-1$, i.e. application of $f_0$ and $f_1$ is actually a multiplication by $a_0, a_1$, respectively. Indeed, assume contrary to our claim that  $f^{n-1}_\omega(y)>1-x_0$ and $\omega_n=1$. Then $f^n_\omega(y)=f_1(f^{n-1}_\omega(y))>f_1(1-x_0)=1-y_0$, hence $r=n$, which is a contradiction. Since $\xo{y}{n}=a^n_\omega y$ and $\xo{x}{n}=a^n_\omega x$ for $n\le r-1$, we have for these $n$'s
\begin{equation}
\label{A7}
|\xo{x}{n}-\xo{y}{n}|=a^n_\omega |x-y|
\end{equation}
But since $f^n_\omega(y)\le 1-y_0<y$ for $n\le r-1$, we have $a_\omega^n <1$ which completes the proof in the case $n\le r-1$.

The only point remaining now is to show that $|\xo{x}{r}-\xo{y}{r}|\le |x-y|$. 
If $\xo{x}{r-1}\ge 1-x_0$, then the statement is obviously true, since both $f_0, f_1$ are contractions on $[1-x_0, x_0]$ and the statement is true for $n=r-1$. We are reduced now to proving $|\xo{x}{r}-\xo{y}{r}|\le |x-y|$ provided that $\xo{x}{r-1}<1-x_0<\xo{y}{r-1}$. Let us consider the function $k\longmapsto |f_1(ky)-f_1(kx)|$ for $k\in \big[\frac{1-x_0}{y}, \frac{1-x_0}{x}\big]$ (this condition is equivalent to say that $1-x_0\in \big[kx, ky\big]$, thus the condition $a_\omega^{r-1}\in \big[\frac{1-x_0}{y}, \frac{1-x_0}{x}\big]$ is equivalent to our case now). We assert that this function is nonincreasing. Indeed,
$$f_1(ky)-f_1(kx)=\big(f_1(ky)-f_1(1-x_0)\big)+\big(f_1(1-x_0)-f_1(kx)\big)$$
$$=a_0\big(ky-(1-x_0)\big)+a_1\big((1-x_0)-kx\big),$$
hence the function is linear with the slope equal to $a_0y-a_1x$ which is negative since $|x-y|<a_1\eta_1$ and (\ref{A3}) holds for $\eta_1$.

We compute now $|f_1(ky)-f_1(kx)|$ for $k=k_0:=\frac{1-x_0}{y}$. We have $|f_1(k_0y)-f_1(k_0x)|=a_1(k_0y-k_0x)=a_1k_0(y-x)$ and $a_1k_0y=f_1(k_0y)= f_1(1-x_0)=1-y_0 \le y$ which implies that $a_1k_0\le 1$. Combining that with the monotonicity of the considered function yields
$$|\xo{x}{r}-\xo{y}{r}|=|f_1(a_\omega^{r-1}y)-f_1(a_\omega^{r-1}x)|\le |f_1(k_0y)-f_1(k_0x)|\le |x-y|$$
which completes the proof of Lemma 1.
\end{proof}

\begin{lem}
\label{lem:A6}
If $x, y\in [1-x_0, x_0]$, $|x-y|<\eta_1$, and $u$ is such that $\xo{y}{n}\le x_0$ for all $n\le u$, then $|\xo{x}{n}-\xo{y}{n}|\le a_1|x-y|$.
\end{lem}
\begin{proof}[Proof of Lemma 2]
The proof is essentially the same as in the case of previous lemma. We define $t$ and $r$ in the same way and assume without loss of generality that $t=1$, $r=u$ (for $n\ge r$ we can apply Lemma \ref{lem:A5}). We again observe that  $\xo{y}{n}=a^n_\omega y$, $\xo{x}{n}=a^n_\omega x$, and $f^n_\omega(y)\le 1-y_0$ for $n\le r-1$. The difference is that $y>1-y_0$ is not true anymore. However, by the definition of $a_1$ we have $1-y_0=a_1(1-x_0)\le a_1y$, thus $a_\omega^n\le a_1$ which proves the assertion for $n\le r-1$ (cf. (\ref{A7})).

If $n=r$ then we have again two cases. If $\xo{x}{r-1}\ge 1-x_0$, then the statement is obviously true, since both $f_0, f_1$ are contractions on $[1-x_0, x_0]$ and the statement is true for $n=r-1$. If $\xo{x}{r-1}<1-x_0$ then $a^r_\omega x=\xo{x}{r}<1-y_0=a_1(1-x_0)\le a_1x$, so $a^r_\omega y- a^n_\omega y\le a_1(y-x)$. Observation that $f^r_\omega(y)<a_\omega^r y$ yields the assertion.
\end{proof}


\begin{proof}[Proof of Proposition 1]
We can define the following infinite sequences of natural numbers
$$t_1:=\min\{n\ge 1: \textrm{$\xo{x}{n}<1-x_0$ or $\xo{y}{n}>x_0$}\},$$
$$
t_{k+1} := \left\{ \begin{array}{ll}
\min\{ n\ge t_k : \xo{y}{n}>x_0 \} & \textrm{if $\xo{x}{t_k}<1-x_0$}\\
\min\{ n\ge t_k : \xo{x}{n}<1-x_0 \} & \textrm{if $\xo{y}{t_k}>x_0$}
\end{array} \right., \quad k\ge 1,
$$
$$
u_k:= \left\{ \begin{array}{ll}
\max\{n\le t_{k+1} : \xo{x}{n}<1-y_0\} & \textrm{if $\xo{x}{t_k}<1-x_0$}\\
\max\{n\le t_{k+1} : \xo{y}{n}>y_0\} & \textrm{if $\xo{y}{t_k}>x_0$}
\end{array} \right., \quad k\ge 1.
$$

\noindent To finish the proof notice that the statement for $n\le u_1$ follows from Lemma \ref{lem:A6} (or its symmetric version) with $u=u_1$. Hence, from the definition of $(u_k)$, the points $f^{u_1}_\omega(x), f^{u_1}_\omega(y)$ satisfy assumptions of Lemma \ref{lem:A5} (or its symmetric version) with $u=u_2-u_1$. We continue in this fashion: for every $k$ the points $f^{u_k}_\omega(x), f^{u_k}_\omega(y)$ satisfy assumptions of Lemma \ref{lem:A5} or its symmetric version with $u=u_{k+1}-u_k$, and the conclusion follows.

To obtain the statement for any $\omega$ observe that for some $k$ we cannot define $t_{k+1}$ and in this case either Lemma \ref{lem:A5} or \ref{lem:A6} applies for  $f^{u_k}_\omega(x), f^{u_k}_\omega(y)$ with arbitrary large $u$.
\end{proof}


\begin{prop}
\label{prop:A4}
There exists $\eta_2>0 $ such that if $x, y\in [1-x_0, x_0]$ and $|x-y|<\eta_2$, then
$$\mbE{x}|\x{x}{n}-\x{y}{n}|\le Lq^n|x-y|$$
for all natural $n$, $L\ge 1$, $q<1$.
\end{prop}

\noindent From now on, $M, \alpha, \varepsilon$, and $p$ always stand for the quantities chosen in the proof of existence of a stationary measure. Fix $x\in (0,1)$ and define
$$A_{x, j}:=\{\omega\in\Sigma : \x{x}{j}(\omega) < \varepsilon\}, \qquad A^{x,j}:=\{\omega\in\Sigma : \x{x}{j}(\omega) > 1-\varepsilon\},$$
$$B_{x, n}:=\bigcap_{j=1}^nA_{x,j}, \qquad B^{x,n}:=\bigcap_{j=1}^nA^{x,j}.$$


\begin{lem}
\label{lem:A1}
If $x<\vare$ then
$$\mbP{x}(B_{x,n})\le  (\varepsilon/x)^\alpha p^n$$
for all $n\ge 0$. The same estimation holds for $\mbP{1-x}(B^{1-x,n})$.
\end{lem}
\begin{proof}
Fix $x\le \varepsilon$ and recall that we write $a^n_\omega=a_{\omega_n}\ldots a_{\omega_1}$ for $n\ge 1$ and $\omega\in\Sigma$.
We first observe that $\mathbb{E}_x\mathds{1}_{B_{x,n-1}}(a_\omega^n)^{-\alpha}\le p^n$. Indeed, by (\ref{A1}) we have
$$\mbE{x}( a_{\omega_{n}}^{-\alpha} | \mathcal F_{n-1})=p_0(\x{x}{n-1}(\omega))a_{0}^{-\alpha}+p_1(\xo{x}{n-1})a_1^{-\alpha}<p$$
provided that $\omega\in B_{x, n-1}$.
Here $(\mathcal F_n)_{n\ge 1}$ stands for the natural filtration in $(\Sigma, \mathcal F)$. Therefore
$$\mathbb{E}_x\mathds{1}_{B_{x,n-1}}(a_\omega^{n})^{-\alpha}=\mbE{x}\bigg(\mathds{1}_{B_{x,n-1}}(a_\omega^{n-1})^{-\alpha} \mbE{x}\big( a_{\omega_{n}}^{-\alpha} | \mathcal F_{n-1} \big)\bigg)<p\mathbb{E}_x\mathds{1}_{B_{x,n-1}}(a_\omega^{n-1})^{-\alpha}.$$
Proceeding by induction yields $\mathbb{E}_x\mathds{1}_{B_{x,n-1}}(a_\omega^n)^{-\alpha}\le p^n$.

Observe that for all $\omega\in\Sigma$ with $\omega\in B_{x,n}$ we have $\x{x}{j}(\omega)=a_\omega^jx$ and, in consequence,
$$B_{x, n}=\{\omega\in\Sigma : a_{\omega}^jx < \varepsilon \ \textrm{for all $j\le n$}\}.$$
The Chebyshev inequality gives now
$$\mbP{x}(B_{x,n})= \mbP{x}(\{\omega\in\Sigma:  \textrm{$a^j_\omega x<\varepsilon$ for all $j \le n$}\})$$
$$\le  \mbP{x}(\{\omega\in\Sigma: (a^n_\omega)^{-1} >x/\varepsilon \}\cap B_{x,n-1}) $$
$$\le (\varepsilon/x)^\alpha\mathbb{E}_x\mathds{1}_{B_{x,n-1}}(a^n_\omega)^{-\alpha}\le (\varepsilon/x)^\alpha p^n$$
which establishes our claim for $\mbP{x}(B_{x,n})$. The same proof works for $\mbP{1-x}(B^{1-x,n})$.
\end{proof}


\begin{lem}
\label{lem:B2}
There exists a point $c\in (1-x_0, x_0)$ such that for every, $h>0$, $\rho>0$ there exist a natural number $n_1$ and $\delta>0$ such that
$$\inf_{x\in[h,1-h]}\mbP{x}(\x{x}{n_1}\in (c-\rho, c+\rho))>0$$
for $x\in [h,1-h]$.
\end{lem}
\begin{proof}
\noindent First of all, by (\ref{A5}) and symmetry we have $f_1([1-x_0, y_0+a_1\eta])\subseteq [1-y_0, x_0]$ and $f_0([1-y_0, x_0])\subseteq [1-x_0, y_0]$. Hence the composition $f_0\circ f_1$ restricted to the interval $[1-x_0, y_0+a_1\eta]$  is a contraction and acts to the interval $[1-x_0, y_0+a_1\eta]$. Let $c$ be the unique attractive fixed point for this composition on $[1-x_0, y_0+a_1\eta]$. For any point $x\in [1-x_0, y_0+a_1\eta]$ and $\rho>0$ there exists $m'$ such that $\mbP{x}(\x{x}{2m'}\in (c-\rho, c+\rho))>0$.

Choose $h>0$. Now it is sufficient to show that for any $x\in [h,1-h]$ there exists a number $m''$ such that $\mbP{x}(\x{x}{2m''}\in [1-x_0, y_0+a_1\eta])>0$. Then $n_1=2m+2m'$ will be desired number, where $m$ is the maximum of $m''$ for $x\in [h,1-h]$. Indeed, the quantity
$$\inf_{x\in [h,1-h]}\min_{(i_1,\ldots,i_{n_1})\in\{0,1\}^{n_1}} p_{i_{n_1}}(\f{i}{n_1-1}{x})\cdot\ldots\cdot p_{i_1}(x)$$
is positive by (A3) and for any $x\in [h, 1-h]$ we can first take a sequence of length $2m''$ with $\mbP{x}(\x{x}{2m''}\in [1-x_0, y_0+a_1\eta])>0$ (which may be less than $2m$) and then apply the composition $f_0\circ f_1$ exactly $m'+(m-m'')$ many times.

We are left to show that for any $x\in [h,1-h]$ there exists $m''$ such that $\mbP{x}(\x{x}{2m''}\in [1-x_0, y_0+a_1\eta])>0$. It is readily seen that there exist $m'''$ and a sequence $(i_1,\ldots, i_{m'''})\in \{0,1\}^{m'''}$ such that 
$z_0:=\f{i}{m'''}{x}\in [1-x_0, y_0+a_1\eta]$. If $m'''$ is even then put $m'''=2m''$. If not then apply $f_0$ to $z_0$. If $f_0(z_0)\ge 1-x_0$ then $m'''+1=2m''$ is a desired number. If not then $f_0(z_0)<1-x_0$, hence $z_1:=f_1\circ f_0(z_0)\ge 1-y_0$. Note that $z_1=a_1a_0z_0>z_0$. We can repeat this procedure and define $z_{n+1}$ while $f_0(z_n)<1-x_0$. This procedure, however, must finish for some $n$, since $z_n=(a_1a_0)^nz_0$ which eventually become greater than $1-y_0$ for some n, which means that $f_1^{-1}(z_n)> 1-x_0$. A contradiction. Let $n$ be the minimal number with $f_0(z_n)\ge 1-x_0$. Then $2m''=m'''+2n+1$ has the desired property.
\end{proof}

\begin{proof}[Proof of Proposition 2]


Let $c$ be the point from Lemma \ref{lem:B2}. Take $\rho>0$ to be any positive number less than distance from $c$ to the boundary points of $[1-x_0, y_0]$. Take $h=\varepsilon$ (recall that $M, \varepsilon, \alpha$ were the numbers given in the proof of existence of the stationary measure; see the comment under Proposition \ref{prop:A4}). Take $n_1$ to be the numbers given in Lemma \ref{lem:B2}. By the continuity of $f_0, f_1$ and the compactness of $[h,1-h]$, there exists $\eta_2>0$ such that if $|x-y|<a_1\eta_2$ then
$$\inf_{x\in[h,1-h]}\mbP{x}(\x{y}{n_1}\in (c-\rho, c+\rho))>0.$$

Let $n_2$ be such that $(a_0)^{n_2}<1/(2a_1)$. Put $m:=n_1+n_2$ and $\xi:=f_0^m(\varepsilon)$ (i.e. $\xi$ is such a number that $\mbP{x}(f^m_\omega(x)\in [\xi, 1-\xi])=1$ for $x\in[\varepsilon, 1-\varepsilon]$). Eventually put
$$\delta:=\inf_{x\in [h,1-h]} \min_{(i_1,\ldots,i_m)\in\{0,1\}^m} p_{i_m}(\f{i}{m-1}{x})\cdot\ldots\cdot p_{i_1}(x)>0.$$

Let us define the following optional times on $\Sigma$ for $u\in (0,1)$:
$$T_1(x):= \min\{n\ge 0: \vare\le \x{x}{n}\le 1-\vare\}+m,$$
$$T_{n+1}(x):=T_n(x)+T_1(\x{x}{T_n(x)})\circ\theta_{T_n(x)},$$
\begin{equation*}
\begin{gathered}
S_1(x):=\min\bigg\{n\ge 1: \forall_{|y-x|<\eta_2} |\x{x}{n}-\x{y}{n}|\le \frac{1}{2a_1}|x-y|, \\ 
\textrm{and} \ \x{x}{n}, \x{y}{n}\in [1-x_0, x_0]\bigg\},
\end{gathered}
\end{equation*}
$$S_{n+1}(x):=S_n(x)+S_1(\x{x}{S_n(x)})\circ\theta_{S_n(x)}$$
$$\tau_n(x):=\max\{k\ge 1: T_k(x)\le n\},$$
$$\sigma_n(x):=\max\{k\ge 1: S_k(x)\le n\},$$
for $x\in (0,1)$. From what has already been proved we conclude that 
$$\mbP{x}(S_1(x)>T_1(x))\le 1-\delta$$
for all $\xi\le x\le 1-\xi$. By the strong Markov property
$$\mbP{x}(S_1(x)>T_{n+1}(x))= \mbE{x}\mbP{x}\big(S_1(x)>T_{n+1}(x) | \mathcal F_{T_n}\big)$$
$$= \mbE{x}\bigg(\mathds{1}_{\{S_1(x)>T_{n}(x)\}}\mbP{\x{x}{T_n(x)}}\bigg( S_1(\x{x}{T_n(x)})\circ\theta_{T_n(x)}>T_1(\x{x}{T_n(x)})\bigg)\bigg)$$
$$\le (1-\delta)\mbP{x}(S_1(x)>T_n(x)),$$
for all $\xi\le x\le 1-\xi$. By induction argument we get
$$\mbP{x}(S_1(x)>T_n(x))\le (1-\delta)^n$$
for such $x$.

By Lemma \ref{lem:A1} there exists $C_1>0$ and $\gamma\in(0,1)$ such that $\mbE{x}e^{\gamma T_1(x)}\le C_1$ for all $x\in [\xi, 1-\xi]$. Induction argument applied below yields
$$\mbE{x} e^{\gamma T_n(x)}=\mbE{x}\bigg(e^{\gamma T_{n-1}(x)}\mbE{x}\big(e^{T_1(f^{T_{n-1}}_\omega(x))}   |\mathcal F_{T_{n-1}} )\bigg)\le C_1 \mbE{x} e^{\gamma T_{n-1}(x)} \le C_1^n$$
for $x\in [\xi, 1-\xi]$ and $n\ge 1$, since $\x{x}{T_n}\in[\xi, 1-\xi]$ for every $n$. Fix $\kappa\in (0,1)$. We have again by the Chebyshev inequality
$$\mbP{x}(\tau_n(x)<\kappa n)\le \mbP{x}(T_{\lfloor\kappa n\rfloor+1}(x)>n)\le C_1^{\lfloor\kappa n\rfloor+1}e^{-\gamma n}$$
for all $x\in [\xi, 1-\xi]$, thus
$$\mbP{x}(S_1(x)>n)\le \mbP{x}(\{S_1(x)>n\}\cap \{\tau_n(x)\ge \kappa n\})+\mbP{x}(\tau_n(x)<\kappa n)$$
$$\le \mbP{x}(S_1(x)>T_{\lfloor \kappa n \rfloor}(x))+C_1^{\lfloor\kappa n\rfloor+1}e^{-\gamma n}\le (1-\delta)^{\lfloor \kappa n  \rfloor}+C_1(C_1^\kappa e^{-\gamma})^n.$$
Choose $\kappa$ such that $C_1^\kappa e^{-\gamma}<1$. By the above we have
$$\mbE{x}e^{\rho S_1(x)}\le e^\gamma \sum_{n=0}^\infty e^{\rho n}\mbP{x}(S_1(x)>n)\le C_2<\infty,$$
for all $x\in[\xi,1-\xi]$ provided that $\rho\in (0,1)$ was chosen sufficiently small. Again, conditioning argument yields
$$\mbE{x}e^{\rho S_n(x)}\le C_2^n.$$
Eventually, using again the Chebyshev inequality we obtain for such $x, y$ and any $\lambda\in (0,1)$,
$$\mbE{x}|\x{x}{n}-\x{y}{n}|=\mbE{x}\mathds{1}_{\{\sigma_n(x,y)< \lambda n\}}|\x{x}{n}-\x{y}{n}| + \mbE{x}\mathds{1}_{\{\sigma_n(x,y)\ge\lambda n\}}|\x{x}{n}-\x{y}{n}|$$
$$\le a_1|x-y|\mbP{x}(S_{\lfloor \lambda n\rfloor(x,y)}> n)+\frac{1}{2^{\lambda n}}|x-y| $$
$$\le a_1 C_2^{\lfloor \lambda n\rfloor}e^{-\rho n}|x-y|+\frac{1}{2^{\lambda n}}|x-y|\le\bigg(a_1(C_2^\lambda e^{-\rho})^n+\frac{1}{2^{\lambda n}}\bigg)|x-y|.$$
Take $\lambda$ such that $C_2^\lambda e^{-\rho}<1$ and put $L=a_1+1$, $q=\max\{C_2^\lambda e^{-\rho}, \frac{1}{2^\lambda}\}<1$. Then by the above we have
$$\mbE{x}|\x{x}{n}-\x{y}{n}|\le Lq^n|x-y|$$
for all natural $n$ which is the desired conclusion.
\end{proof}

\begin{proof}[Proof of uniqueness]
Throughout the proof $p_{i_1,\ldots,i_n}(x)$ stands for
$$ p_{i_n}(\f{i}{n-1}{x})\cdot\ldots\cdot p_{i_1}(x).$$
First observe that for any $x\in (0,1)$ there exists a finite sequence $(i_1,..,i_l)\in \{0,1\}^l$ for some $l$ such that $\f{i}{l}{x}\in [1-x_0, y_0]$ which implies that the topological support of any $P$-invariant measure $\mu$ must have nonempty intersection with $[1-x_0, y_0]$. Further, $f_1\circ f_0([1-x_0, y_0])\subseteq (1-x_0, y_0)$ by (\ref{A5}) and $f_1\circ f_0$ is a contraction on the interval $[1-x_0, y_0]$, hence  this composition has exactly one attractive fixed point $c \in (1-x_0, y_0)$. Combining these facts yields $c\in \Gamma_\mu$ for all $P$-invariant measures $\mu$, where $\Gamma_\mu$ denotes the topological support of this measure. The proof is completed by showing that the family $(U^n\varphi)$ is equicontinuous at $c$ for any Lipschitz $\varphi$. Indeed, if there exist at least two different ergodic invariant measures $\mu_1$, $\mu_2$, then there exists a Lipschitz function $\varphi$ such that $\big|\int\varphi d\mu_1-\int\varphi d\mu_2\big|\ge \delta$ for some $\delta>0$. We consider the averages $\frac 1 n (\varphi(x)+U\varphi(x)+\ldots+U^{n-1}\varphi(x))$ which must differ from  $\frac 1 n (\varphi(c)+U\varphi(c)+\ldots+U^{n-1}\varphi(c))$ at most $\delta/2$, provided that $x$ is sufficiently close to $c$. On the other hand, $c\in\Gamma_{\mu_1}\cap\Gamma_{\mu_2}$, therefore in any neigbourhood of $c$ we can find points $x_1, x_2$ such that considered averages tend to $\int\varphi d\mu_1, \int\varphi d\mu_2$, respectively, which is a contradiction.

We are going to show that $(U^n\varphi)$ is equicontinuous at any point of $(1-x_0, x_0)$. Take $x\in (x_0, 1-x_0)$ and $\delta>0$. Take $n_0$ such that $\sum_{n=n_0}^\infty 2\beta(Lq^n)<\frac{\delta}{6\|\varphi\|_\infty}$ (by (A2)) and $Lq^n\le \frac{\delta}{3\textrm{Lip}(\varphi)}$ for $n\ge n_0$, where $\textrm{Lip}(\varphi)$ denotes the Lipschitz constant of $\varphi$. By Theorem 8 on the page 45 in \cite{Lorentz_Approximation_of_functions} there exists a concave function $\beta^*$ with $\beta(t)\le \beta^*(t)\le 2\beta(t)$. Thus we have $\sum_{n=n_0}^\infty \beta^*(Lq^n)<\frac{\delta}{3\|\varphi\|_\infty}$.

Take $y$ such that $|x-y|<\eta_2$ and
\begin{equation}
\label{A6} 
\sum \bigg| p_{i_1,\ldots, i_{n_0}}(x)-p_{i_1,\ldots, i_{n_0}}(y)\bigg|<\frac{\delta}{3\|\varphi\|_\infty},
\end{equation}
where the summation is over all finite sequences $(i_1,\ldots,i_{n_0})\in\{ 0,1\}^{n_0}$. It is satisfied provided that $|x-y|$ is less than, say, $d>0$. Then for $n\ge n_0$ we have
$$|U^n\varphi(x)-U^n\varphi(y)|$$
$$\le\sum p_{i_1,\ldots, i_n}(x)\bigg|\varphi(\f{i}{n}{x})-\varphi(\f{i}{n}{y})\bigg|$$
$$+\bigg|p_{i_1,\ldots,i_n}(x)-p_{i_1,\ldots, i_n}(y) \bigg| \|\varphi\|_\infty,$$
where the summation is over all finite sequences $(i_1,\ldots,i_n)\in\{ 0,1\}^n$. The first term is bounded by $\textrm{Lip}(\varphi)\mbE{x}|\x{x}{n}-\x{y}{n}|$, and the second term divided by $\|\varphi\|_\infty$ is bounded by
$$=\sum_{i_1,\ldots, i_{n}} \bigg| p_{i_n}(\f{i}{n-1}{x})-p_{i_n}(\f{i}{n-1}{y})\bigg|\cdot p_{i_1,\ldots, i_{n-1}}(x)$$
$$+\sum_{i_1,\ldots, i_{n}} p_{i_n}(\f{i}{n-1}{y})\bigg | p_{i_1,\ldots, i_{n-1}}(x)- p_{i_1,\ldots, i_{n-1}}(y)\bigg|$$
$$\le 2\mbE{x}\beta^*(|\x{x}{n}-\x{y}{n}|)+\sum_{i_1,\ldots, i_{n-1}}\bigg|p_{i_1,\ldots, i_{n-1}}(x)-p_{i_1,\ldots, i_{n-1}}(y) \bigg|.$$
The modulus of continuity $\beta^*$ is concave, therefore by the Jensen inequality we have
$$\le 2\beta^*(Lq^n)+\sum_{i_1,\ldots, i_{n-1}}\bigg|p_{i_1,\ldots, i_{n-1}}(x)-p_{i_1,\ldots, i_{n-1}}(y) \bigg|.$$
Continuing this procedure while $n> n_0$ and using (\ref{A6}) yields
$$\le \sum_{i=n_0}^n 2\beta^*(Lq^n)+\sum_{i_1,\ldots, i_{n_0}} \bigg| p_{i_1,\ldots, i_{n_0}}(x)-p_{i_1,\ldots, i_{n_0}}(y)\bigg|<\frac{\delta}{3\|\varphi\|_\infty}+\frac{\delta}{3\|\varphi\|_\infty}.$$
Again by the definition of $n_0$ we have
$$|U^n\varphi(x)-U^n\varphi(y)|< \textrm{Lip}(\varphi)\mbE{x}|\x{x}{n}-\x{y}{n}|$$
$$+\|\varphi\|_\infty\frac{\delta}{3\|\varphi\|_\infty}+\|\varphi\|_\infty\frac{\delta}{3\|\varphi\|_\infty}< \delta$$
for all $n$ and $y$ with $|x-y|<d$. Therefore $(U^n\varphi)$ is equicontinuous at any $x\in[1-x_0, x_0]$ which proves the uniqueness of the invariant measure $\mu_*$.
\end{proof}

\section{The proof of Theorem 2}

\begin{lem}
\label{lem:B1}
There exists $0<h<1/2$ such that for all $0<\xi<1/2$ there exists $n_0$ such that $P^n\delta_x([h,1-h])=\mbP{x}(\x{x}{n}\in[h,1-h])\ge 1/2$  for all $x\in[\xi,1-\xi]$  and $n\ge n_0$.
\end{lem}
\begin{proof}
Recall that $M$, $\alpha$, $\vare$ are the quantities given in the proof of existence of a stationary measure. The class $\mathcal P_{M,\alpha}$ is $P$-invariant. Take $h>0$ such that $Mh^\alpha<1/8$. The definition of $M$ yields the relation $M\varepsilon^\alpha\ge 1$ which implies clearly $P^n\delta_x\in\mathcal P_{M,\alpha}$ and thus $P^n\delta_x\big([h,1-h])\ge 3/4$ for every $x\in [\vare, 1-\vare]$ and $n\ge 0$.

Let $n_0$ be such that $ (\varepsilon/\xi)^\alpha p^{n}<1/8$ for $n\ge n_0$. Take $x\not\in[\vare,1-\vare]$ and $x\in[\xi, 1-\xi]$. Denote by $T$ the time of the first visit of $x$ in $[\vare, 1-\vare]$. Then by the strong Markov property, Lemma \ref{lem:A1} and the first part of the proof we have
$$\mbP{x}\big( \x{x}{n}\not\in [h,1-h]\big)$$
$$\le \sum_{k=1}^{n} \mbP{x}( \x{x}{n}\not\in [h,1-h] | T=k)\mbP{x}(T=k)+ \mbP{x}(T>n)<1/2,$$
for $n\ge n_0$, since  $\mbP{x}(T>n)\le (\varepsilon/x)^\alpha p^{n}\le (\varepsilon/\xi)^\alpha p^{n}<1/8$ for $x\in[\xi, 1-\xi]$.
\end{proof}

\noindent From now on, $h$ denotes the quantity given in Lemma \ref{lem:B1}.


\begin{lem}
\label{lem:B3}
For every $\rho>0$ there exists $\delta>0$  such that for every $\xi>0$ there exists a natural number $m$ such that for all $n\ge m$ we have
$$\inf_{x\in[\xi, 1-\xi]} \mbP{x}(\x{x}{n}\in (c-\rho, c+\rho))\ge \delta.$$
\end{lem}
\begin{proof}
Fix $\xi$. Let $n_0$ be the number given in Lemma \ref{lem:B1}. Let $n_1, \zeta$ be the numbers given in Lemma \ref{lem:B2}. Let $m:=n_1+n_0$ and $\delta:=\zeta/2$. Take $n\ge m$. Then $n-n_1\ge n_0$, thus
$$ \mbP{x}(\x{x}{n}\in (c-\rho, c+\rho))= \mbP{x}\big(\x{x}{n}\in (c-\rho, c+\rho) | \x{x}{n-n_1}\in [h,1-h] \big)$$
$$\cdot\mbP{x}(\x{x}{n-n_1}\in [h,1-h] )\ge 1/2\cdot\zeta>0$$
for every $x\in[\xi,1-\xi]$ by Lemma \ref{lem:B2}.
\end{proof}

\noindent We are in position to show Theorem \ref{thm:C}. The idea is to apply the lower bound technique (cf. Theorem 4.1 in \cite{Lasota}).

\begin{proof}[Proof of Theorem 2]
Take $x<y$, $\lambda>0$ and a Lipschitz function $\varphi$. By the equicontinuity of $(U^n\varphi)$ at $c$ there exists $\rho>0$ such that
\begin{equation}
\label{B1}
|U^n\varphi(u)-U^n\varphi(v)|<\lambda \ \textrm{for $n\ge 1$ and $u,v\in (c-\rho, c+\rho)$.}
\end{equation}
Define
$$A_n:=\{ (\omega, \omega')\in\Sigma\times\Sigma : c-\rho< f^n_\omega(x)<f^n_{\omega'}(y) <c+\rho \}.$$ 
By Lemma \ref{lem:B2} there exist $m_1$, $\delta>0$ with $\mbP{x}\otimes\mbP{y}(A_{m_1})\ge \delta^2$. Put $\xi_1:=\min\{f^{m_1}_0(x), 1-f_1^{m_1}(y)\}$. Then $\xi_1\le f^{m_1}_\omega(x)<f^{m_1}_{\omega'}(y)\le 1-\xi_1$ for all $(\omega, \omega')\in\Sigma\times\Sigma$. Once again, by Lemma \ref{lem:B3} there exists $m_2$ such that $\mbP{x}\otimes\mbP{y}(A_{m_2})\ge \delta^2$. Put $\xi_2:=\min\{f^{m_1+m_2}_0(x), 1-f_1^{m_1+m_2}(y)\}$. Obviously, $\xi_2\le f^{m_1+m_2}_\omega(x)<f^{m_1+m_2}_{\omega'}(y)\le 1-\xi_2$ for all $(\omega, \omega')\in\Sigma\times\Sigma$.

We continue in this fashion to construct a sequence $m_1, m_2,\ldots$ such that
$$\mbP{x}\otimes\mbP{y}(A_{m_k})\ge \delta^2$$
for all $n$'s. It is easy to check that
\begin{equation}
\label{B2}
\mbP{x}\otimes\mbP{y}(B_n)\le (1-\delta^2)^n,
\end{equation}
where
$$B_n:=\bigcap_{k=1}^n \Sigma\times\Sigma\setminus A_{m_k}.$$
Hence for $n\ge m_k$ we get
$$U^n\varphi(x)-U^n\varphi(y)=\int\int_{\Sigma\times\Sigma}\big(\varphi(f^n_\omega(x))-\varphi(f^n_{\omega'}(y))\big)\mbP{x}(\textrm{d$\omega$})\otimes\mbP{y}(\textrm{d$\omega'$)}$$
$$=\sum_{j=1}^k \int\int_{A_{m_j}}\mbE{x,y}\big(\varphi(f^n_\omega(x))-\varphi(f^n_{\omega'}(y))| \mathcal F_{m_j}\big)\mbP{x}(\textrm{d$\omega$})\otimes\mbP{y}(\textrm{d$\omega'$)}$$
$$+\int\int_{B_k}\big(\varphi(f^n_\omega(x))-\varphi(f^n_{\omega'}(y))\big)\mbP{x}(\textrm{d$\omega$})\otimes\mbP{y}(\textrm{d$\omega'$)}$$
$$=\sum_{j=1}^k \int\int_{A_{m_j}}\big(U^{n-{m_j}}\varphi(f^{m_j}_\omega(x))-U^{n-{m_j}}\varphi(f^{m_j}_{\omega'}(y))\big)\mbP{x}(\textrm{d$\omega$})\otimes\mbP{y}(\textrm{d$\omega'$)}$$
$$+\int\int_{B_k}\big(\varphi(f^n_\omega(x))-\varphi(f^n_{\omega'}(y))\big)\mbP{x}(\textrm{d$\omega$})\otimes\mbP{y}(\textrm{d$\omega'$)}.$$
By (\ref{B1}), (\ref{B2}), and the definition of $A_{m_j}$'s eventually we have
$$|U^n\varphi(x)-U^n\varphi(y)|\le \lambda + (1-\delta^2)^k<2\lambda$$
provided that $k$ was sufficiently large. Therefore
$$\lim_{n\to\infty} |U^n\varphi(x)-U^n\varphi(y)|\to 0$$
for every $x,y\in (0,1)$. If $\mu_*$ is the stationary probability measure and $\nu\in \mathcal M_1$, then for any Lipschitz function $\varphi$ we obtain
$$\bigg|\int_{(0,1)}\varphi(x)P^n\nu(\textrm{d$x$})-\int_{(0,1)}\varphi(y)\mu_*(\textrm{d$y$})\bigg|
=\bigg|\int_{(0,1)}U^n\varphi(x)\nu(\textrm{d$x$})-\int_{(0,1)}U^n\varphi(y)\mu_*(\textrm{d$y$})\bigg|$$
$$\le \int\int_{(0,1)\times(0,1)}\big|U^n\varphi(x)-U^n\varphi(y)|\nu(\textrm{d$x$})\otimes\mu_*(\textrm{d$y$})\to 0$$
by the Lebesgue Convergence Theorem. Thus $P^n\nu\to\mu_*$ weakly-$\ast$ for every $\nu\in\mathcal M_1$ which is our assertion.
\end{proof}

%
%
\section{The proof of Theorem 3} 

\noindent Let $c$ be the number given in Lemma \ref{lem:B2}. Recall that $c$ is the unique attractive fixed point of the composition $f_0\circ f_1$ on $(1-x_0, y_0)$.  For any $\rho>0$ we will write $S_\rho(x)$ for the time of the first visit of the process $(f^n_\omega(x))$ in $(c-\rho, c+\rho)$.

\begin{lem}
\label{lem:C1}
If $\rho>0$, $x\in (0,1)$, then $S_\rho(x)$ is finite $\mbP{x}$-a.s.
\end{lem}
\noindent We omit the proof, since it is an easy consequence of Lemma \ref{lem:A1} and Lemma \ref{lem:B2}.

\begin{lem}
\label{lem:C2}
Let $q<1$, $L\ge 1$, $\eta_2>0$, be the quantities given in Proposition \ref{prop:A3}. Let $x,y\in[1-x_0, x_0]$ be such that $|x-y|<\eta_2$. If $q<r<1$ then for every $\lambda>0$ there exist a natural $n_\lambda$ and a measurable set $\widetilde{\Sigma}\subseteq\Sigma$ such that $\mbP{y}(\widetilde{\Sigma})>1-\lambda$ and
$$ |f^n_\omega(x)-f^n_\omega(y)|< r^n $$
for every $\omega\in\widetilde{\Sigma}$ and $n\ge n_\lambda$.
\end{lem}
\begin{proof}
This is an immediate consequence of the Chebyshev inequality and the Borel-Cantelli lemma. Indeed,
$$\mbP{y}( \{\omega\in\Sigma : |f^n_\omega(x)-f^n_\omega(y)|\ge r^n\} )\le \mbE{y}|f^n_\omega(x)-f^n_\omega(y)| r^{-n}\le L(q/r)^n,$$
therefore $\{|f^n_\omega(x)-f^n_\omega(y)|\ge r^n\}$ occurs only finitely many times $\mbP{x}$-a.s which completes the proof.
\end{proof}

\noindent The following lemma is proven in \cite{Elton}, Lemma 3. For the convenience of the reader we rewrite the proof here.

\begin{lem}
\label{lem:C3}
There exists $\rho>0$ such that for every $x,y\in (c-\rho, c+\rho)$ the measures $\mbP{x}$, $\mbP{y}$ on $\Sigma$ are absolutely continuous.
\end{lem}
\begin{proof}
Put $\rho:=\eta_2/2$. We have $\sum_{k=1}^\infty \beta(r^k)<\infty$, since $p_0, p_1$ are Dini continuous (let us recall that $\beta$ stands for the modulus of continuity of $p_0$ and $p_1$. Take $\delta$ such that $\delta<p_0(z)<1-\delta$ for every $z\in (0,1)$, by assumption (A3).

Fix $x,y\in(c-\rho,c+\rho)$ and a measurable set $E$ with $\mbP{x}(E)=0$. Take $\lambda>0$ and $q<r<1$, where $q$ is given in Lemma \ref{prop:A4}. We will show that $\mbP{y}(E)<2\lambda$. Let $\widetilde{\Sigma}$ and $n_\lambda$ be given in Lemma \ref{lem:C2}. Let $m\ge n_\lambda$ be such that $\sum_{k=m+1}^\infty \beta(r^k)<\lambda/2$.

Put $\Sigma_*:=\bigcup_{k=1}^\infty \{0,1\}^k$ and let $\Xi \subseteq \Sigma_*$ be a countable set such that $E\subseteq \bigcup_{\textbf{i}} C_{\textbf i\in\Xi}$ and $\mbP{x}(\bigcup_{\textbf{i}\in\Xi} C_{\textbf i})<(\lambda/2)(\delta/(1-\delta)^m$, where $C_\textbf{i}$ denotes the cylinder set in $\Sigma$ corresponding to the finite sequence $\textbf i\in \Xi$. Moreover, we assume $C_\textbf{i}$ to be pairwise disjoint for $\textbf{i}\in\Xi$. Let $Q_n:=\{(i_1, i_2,\ldots ) \in\Sigma : \textrm{$|f^k_\omega(x)-f^k_\omega(y)|< r^k$ for $m\le k \le n$}\}$ for $n\ge m$ and $Q_n:=\Sigma$ for $n<m$, $Q:=\bigcap_{n=1}^\infty Q_n$. By Lemma \ref{lem:C2} we have $\mbP{y}(\Sigma\setminus Q)\le \mbP{y}(\Sigma\setminus \widetilde{\Sigma})<\lambda$. Take $n\ge m$, $(i_1, i_2, \ldots)\in Q_n$. We obtain the following estimation
$$p_{i_1,\ldots,i_n}(y)\le p_{i_1,\ldots,i_n}(x)\bigg(\frac{1-\delta}{\delta}\bigg)^m \prod_{k=m+1}^n\bigg(1+\frac{|p_{i_1,\ldots,i_k}(y)-p_{i_1,\ldots,i_k}(x)|}{p_{i_1,\ldots,i_k}(x)}\bigg)$$
$$\le  p_{i_1,\ldots,i_n}(x)\bigg(\frac{1-\delta}{\delta}\bigg)^m \prod_{k=m+1}^n\bigg(1+\frac{\beta(r^k)}{\delta}\bigg).$$
One can show easy by induction the following claim: if $r_1, r_2, \ldots$ are positive numbers such that $\sum_{k=1}^\infty r_k<1/2$, then $(1+r_1)\cdot\ldots\cdot(1+r_k)\le 1+2(r_1+\ldots+r_k)$. Application of this claim yields
$$\prod_{k=m+1}^n\bigg(1+\frac{\beta(r^k)}{\delta}\bigg)\le 1+2\sum_{k=m+1}^\infty \frac{\beta(r^k)}{\delta}\le 2,$$
and thus
$$p_{i_1,\ldots,i_n}(y)\le  2\bigg(\frac{1-\delta}{\delta}\bigg)^m p_{i_1,\ldots,i_n}(x)$$
for $n\ge m$. If $n<m$ then this holds trivially for any $\omega\in\Sigma$.

Take $\textbf i =(i_1,\ldots, i_l) \in \Xi$. If $Q$ and $C_\textbf{i}$ are not disjoint then also $Q_l$ and $C_\textbf{i}$ are not disjoint, hence we have be the above estimation
$$\mbP{y}(Q\cap C_\textbf{i})\le \mbP{y}(Q_l\cap  C_\textbf{i})=p_{i_1,\ldots,i_l}(y)$$
$$\le  2\bigg(\frac{1-\delta}{\delta}\bigg)^m p_{i_1,\ldots,i_l}(x)= 2\bigg(\frac{1-\delta}{\delta}\bigg)^m \mbP{x}(C_\textbf{i}).$$
Moreover,
$$\mbP{y}\bigg((\Sigma\setminus Q) \cap \bigcup_{\textbf{i}\in\Xi} C_{\textbf i}\bigg)<\lambda.$$
Recall here that the cylinders $C_\textbf{i}$, $\textbf i \in\Xi$ are disjoint. Combining that with two above inequalities yields
$$\mbP{y}(E)\le \mbP{y}\bigg(\bigcup_{\textbf{i}\in\Xi} C_{\textbf i} \bigg) = \mbP{y}\bigg((\Sigma\setminus Q) \cap\bigcup_{\textbf{i}\in\Xi} C_{\textbf i} \bigg)+\mbP{y}\bigg( Q \cap\bigcup_{\textbf{i}\in\Xi} C_{\textbf i} \bigg)$$
$$<\lambda+\sum_{\textbf i \in \Xi}\mbP{y}(Q\cap C_{\textbf i})\le \lambda+\sum_{\textbf i \in \Xi}2\bigg(\frac{1-\delta}{\delta}\bigg)^m \mbP{x}(C_\textbf{i})$$
$$=\lambda+2\bigg(\frac{1-\delta}{\delta}\bigg)^m \mbP{x}\bigg(\bigcup_{\textbf{i}\in\Xi} C_{\textbf i} \bigg)<2\lambda$$
which is the desired assertion.
\end{proof}

\begin{proof}[Proof of Theorem 3]
Let $\varphi$ be any Lipschitz function. The statement for any continuous function follows from the density of the set of Lipschitz functions in $C\big((0,1)\big)$ with the supremum norm. Let $\rho$ be given in Lemma \ref{lem:C3}. There exists $y\in (c-\rho, c+\rho)$ such that
$$\frac{\varphi(y)+\ldots+\varphi(f_\omega^{n-1}(y))}{n}\xrightarrow[n\to\infty]{} \int \varphi \textrm{d$\mu_*$}$$
for $\mbP{z}$-a.e. $\omega\in\Sigma$ where $\mu_*$ is the unique $P$-invariant measure. It follows by the fact that $c$ is in the support of $\mu_*$ (see the beginning of the proof of uniqueness) and by the Birkhoff Ergodic Theorem. Take any $z\in (c-\rho, c+\rho)$. For $\mbP{z}$-a.e. $\omega\in\Sigma$ there exists $n(\omega)$ such that $|f^n_\omega(z)-f^n_\omega(y)|\le r^n$ for $n\ge n(\omega)$, where $q<r<1$ by Lemma \ref{lem:C2}. Therefore
$$\bigg|\frac{\varphi(y)+\ldots+\varphi(f_\omega^{n-1}(y))}{n}-\frac{\varphi(z)+\ldots+\varphi(f_\omega^{n-1}(z))}{n}\bigg|\xrightarrow[n\to\infty]{} 0$$
for $\mbP{z}$-a.e. $\omega\in\Sigma$. By Lemma \ref{lem:C3} the measures $\mbP{z}$, $\mbP{y}$ on $\Sigma$ are absolutely continuous. Hence
$$\frac{\varphi(z)+\ldots+\varphi(f_\omega^{n-1}(z))}{n}\xrightarrow[n\to\infty]{} \int \varphi \textrm{d$\mu_*$}$$
for $\omega\in D_z$, where $D_z\subseteq \Sigma$ is certain measurable set with $\mbP{z}(D_z)=1$.

To complete the proof fix any $x\in (0,1)$ and observe that by Lemma \ref{lem:C1} one can find a set $\Xi\subseteq\Sigma_*=\bigcup_{k=1}^\infty \{0,1\}^k$ such that $f_{i_l}\circ\ldots\circ f_{i_1}(x)\in (c-\rho, c+\rho)$ for $\textbf i=(i_1,\ldots, i_l)\in\Xi$, $\mbP{x}\bigg(\bigcup_{\textbf i \in \Xi} C_\textbf{i}\bigg)=1$ and, moreover, the family of cylinder sets $C_\textbf{i}$, $\textbf i \in\Xi$ is disjoint. Then for
$$C:=\bigcup_{(i_1,\ldots, i_l)=\textbf i \in \Xi} (i_1,\ldots, i_l)\times D_{f_{i_l}\circ\ldots\circ f_{i_1}(x)}$$
we have $\mbP{x}(C)=1$ and 
$$\frac{\varphi(x)+\ldots+\varphi(f_\omega^{n-1}(x))}{n}\xrightarrow[n\to\infty]{} \int \varphi \textrm{d$\mu_*$}$$
for every $\omega\in C$.
\end{proof}

\section{Open problems}
One cannot hope to prove uniqueness of a stationary measure with continuous probabilities without any additional assumptions on them. Indeed, in \cite{Stenflo} it is proved that for two affine contractions of the interval there exists continuous probability functions $p_1, p_2$ such that there is no uniqueness of a stationary measure for the corresponding Markov chain. It is reasonable to assume (A2) from two reasons. The first one is that it is exactly what we need to deduce from Proposition \ref{prop:A4} the equicontinuity of the family $(U^n)$ at some point of the interval $[1-x_0, x_0]$. The second is that this assumption appears also in papers \cite{BDEG} and \cite{Lasota}, hence it seems to be natural.

The assumption (A3) is not very restrictive and without that the situation is more complicated. For example, admission of vanishing probabilities may easily create invariant intervals, i.e. disjoint intervals such that probability of getting from one to the another is zero. We used this assumption for example in the proof of Lemma \ref{lem:B2}.

The assumption (A4) was crucial to ensure the existence of a stationary measure and to show Lemma \ref{lem:A1} which was a key ingredient in the proof of Proposition \ref{prop:A4}. Without it, different situations may happen. For example, $X^x_n\to 0$ a.s., if $\Lambda_0<0$ and $\Lambda_1\ge 0$ (see \cite{Homburg}).

The assumption (A1) was also important in our reasoning. Essentially it was used only in the proof of Proposition \ref{prop:A3}, but we are not able to show Proposition \ref{prop:A4} without Proposition \ref{prop:A3}. Proposition \ref{prop:A3} is generally not true in the case of all Alsed\`a-Misiurewicz systems, so the following question is natural:

1. It is not possible (in general) to show Proposition \ref{prop:A3} without (A1). However, is it possible to show Proposition \ref{prop:A4} without this assumption? If not, then it is possible to show uniqueness of a stationary measure?

Our method of proving Theorem \ref{thm:C} does not provide any rate of convergence of $U^n\varphi$ to $\int \varphi \textrm{d$\mu$}$. However, if one assume that probabilities are Lipschitz continuous then we may expect that rate of convergence is exponential (see the main result in \cite{Sleczka}), therefore sufficiently fast to provide the Central Limit Theorem (see \cite{Gordin_Lifsic}, \cite{Maxwell-Woodroofe}).

2. Does the Central Limit Theorem hold for our Markov chains provided that probabilities are Lipschitz continuous?

The last question is connected with paper \cite{Baranski-Spiewak}. The authors prove there that under some assumptions, if $y_0<1/2$ then there are invariant Cantor sets for the iterated function system $(f_0, f_1)$. However, nothing is known in the case $y_0>1/2$ which is our case.

3. If the Alsed\`a-Misiurewicz system satisfies (A1) then is it necessarily minimal? 

\section{Acknowledgements}
We are grateful to Tomasz Szarek for communicating the problem and reading the manuscript.


\bibliographystyle{plain}
\bibliography{Bibliography}

\end{document}